\documentclass[12pt]{amsart}
\usepackage{amsmath}
\usepackage{amsfonts}
\usepackage{amssymb}
\usepackage{subfigure}
\usepackage{graphicx}
\usepackage{color}
\usepackage{comment}

\numberwithin{figure}{section}
\newtheorem{lemma}{Lemma}[section]

\newtheorem{definition}[lemma]{Definition}
\newtheorem{corollary}[lemma]{Corollary}

\newtheorem{proposition}[lemma]{Proposition}
\newtheorem{remark}[lemma]{Remark}
\newtheorem{theorem}[lemma]{Theorem}
\numberwithin{equation}{section}
\def\R{{\mathcal R}}
\renewcommand{\H}{{\mathcal{H}}}
\newcommand{\C}{{\mathbb{C}}}

\newcommand{\pbox}{\hfill$\Box$\\}

\begin{document}
\title[]{Affine density, von Neumann dimension and a problem of Perelomov}
\author{Lu\'{\i}s Daniel Abreu and Michael Speckbacher}
\address{\noindent (L.D.A. and M.S.) Acoustics Research Institute, Austrian Academy
of Sciences, Wohllebengasse 12-14, A-1040, Vienna, Austria.\medskip   \newline
{ (M.S.)} Katholische Universit\"at Eichst\"att-Ingolstadt,
Mathematisch-Geographi- sche Fakult\"at, Ostenstra\ss e 26, 85071
Eichst\"att, Germany.}
\thanks{L. D. A. and M. S. were supported by the Austrian Science Fund (FWF)
via the projects (P 31225-N32), and (J-4254) respectively. The authors would also like to thank the anonymous referee for the valuable comments and suggestions.}

\begin{abstract}
We provide a solution to Perelomov's 1972 problem concerning the existence of a phase transition (known in signal analysis as 'Nyquist rate') determining the basis properties of certain affine coherent states labelled by Fuchsian groups. As suggested by Perelomov, the transition is given according to the hyperbolic volume of the fundamental region. The solution is a more general form (in   phase space) of the $PSL(2,\mathbb{R})$ variant of a 1989 conjecture of Kristian Seip about wavelet frames,  where the same value of `Nyquist rate' is obtained as the trace of a certain localization operator. The proof consists of first connecting the problem  to the theory of von Neumann algebras, by introducing a new class of projective representations of $PSL(2,\mathbb{R})$ acting on non-analytic Bergman-type spaces. We then adapt to this setting a new method for computing von Neumann dimensions, due to Sir Vaughan Jones. Our solution contains necessary conditions in the form of a `Nyquist rate' dividing frames from Riesz sequences of coherent states and sampling from interpolating sequences. They hold for an infinite sequence of spaces of polyanalytic functions containing the eigenspaces of the
Maass operator and their orthogonal sums. Within mild boundaries, we show that our result is best possible, by characterizing our sequence of function spaces as the only invariant spaces under the non-analytic $PSL(2,\mathbb{R})$-representations.
\end{abstract}

\maketitle

\section{Introduction}

When the measure of successive dilations of a given compact set displays
exponential growth, one often faces severe obstructions in problems
involving basis properties or zeros of functions with membership in the
associated function spaces. The exponential growth is characteristic for the
disc and half-plane hyperbolic space models, for the affine group, and for
more general hyperbolic spaces and non-unimodular groups. The obstructions
posed by the exponential growth have been removed for Bergman spaces in \cite%
{PerelomovCompleteness,Seip2}, but the methods fully depend on tools for analytic
functions. 

In the present paper, we  study completeness and basis properties of
non-analytic functions indexed by Fuchsian groups, naturally associated with
Maass forms and wavelet representations (coefficients of the non-unimodular
affine group). Our results  provide a solution of Perelomov's 1972 problem 
\cite{PerelomovQuestion} about basis properties of discrete coherent states for a whole class of  $PSL(2,\mathbb{R})$-invariant
coherent states, and include a proof that the solution is best possible within natural
assumptions, by characterizing all function spaces which are invariant under the action of a new class of projective $PSL(2,\mathbb{R})$-representations, introduced in this paper for non-analytic Bergman spaces. The arguments, based on the interplay of ideas from signal
analysis, number theory, and von Neumann algebras, build on a recent method
of Sir Vaughan Jones \cite{Jones}, to whose memory this paper is dedicated.

\subsection{The Nyquist rate}

There is a general `folk theorem'\ in signal analysis claiming the existence
of a critical density, the so called `Nyquist rate',\ which must be
exceeded for the complete representation of signals. The `stable' version of
the problem, where a lower bound on the $\ell ^{2}$-norm of the discrete
coefficients leads to the formulation in terms of frames achieved a
notorious mainstream status at the crossroads of pure and applied
mathematics. For spaces of functions whose Fourier transform is supported in
a general measurable set,\ a theorem of Landau \cite{la67}, and a version for
radial functions \cite{AB}  turn these heuristics into a precise result. In
Gabor analysis, a similar condition for lattices in the time-frequency plane
follows from the work of Rieffel \cite{Rieffel} and has been extended to
general sets by Ramanathan and Steger \cite{RamStee,DensityGaborCHeil}
(interestingly enough, Rieffel's proof depends on arguments using the
coupling constant of von Neumann algebras; developments with a more explicit
link to signal representations can be found in \cite[Section 6]{DLL}, or in 
\cite{Bekka,Han}). An overview of these and other results is given in \cite%
{Heil} as a historical recollection, and in \cite{Survey} and \cite{Sampling1} as a technical survey exploring the
limits of existing operator theory methods. It is  clear from  these contributions that groups of exponential growth like the `$%
ax+b$'-group are beyond the scope of the methods available.

\subsection{The affine density problem}

No known concept of density has been able to capture the `folk theorem' in
the case of wavelets with positive frequencies, defined as the coefficients
of the `$ax+b$'-group representation acting on a function $\psi $ on the
Hardy space $\mathcal{H}^{2}(\mathbb{C}^{+})$ of analytic functions on the upper half-plane $\C^+$ with finite norm
$$
\|f\|_2=\sup_{s>0}\left(\int_\mathbb{R} |f(x+is)|^2dx\right)^{1/2},
$$ and originally introduced in
physics as \emph{affine coherent states}. The origin of the problem can be
traced back to the work of Perelomov in the early 70's \cite%
{PerelomovCompleteness}. In \cite{PerelomovCompleteness}, a complete
description of the completeness properties of affine coherent states indexed
by Fuchsian groups with automorphic cusp forms has been obtained for the
vector of least weight. This corresponds to the Cauchy wavelet $%
\psi_0^\alpha(t)=(t+i)^{-\alpha -1}$, whose phase space is a weighted Bergman space. The
question of a more general result has been raised in \cite{PerelomovQuestion}%
.

For band-limited functions with band-width $W$, a comparison between the number of eigenvalues
of the Toeplitz operator $T^{I}_{\Omega}=P_W\chi_I P_W$, where $\chi_I$ is the indicator function of a
region $I\subset \mathbb{R}$ and $P_W$ is the orthogonal projection onto the Paley-Wiener space of band-limited functions, and the number of sampling and interpolation points on
the same region suffices to establish that the Nyquist rate is given by $%
\frac{\text{trace}\left( T^{I}_\Omega\right) }{|I|}$, see \cite{la67}. Let $W_\psi$ denote the wavelet transform with mother wavelet $\psi$ as defined in \eqref{eq:def-wt} and set  $\mathcal{W}_\psi:=W_\psi(\H^2({\mathbb{C}}^+))$. In \cite%
{SeipConjecture}, reasoning in analogy with the aforementioned situation,
Seip conjectured that the same thing should be true for
wavelets, suggesting that the quantity $\frac{C_{\psi }}{\Vert \psi \Vert
_{2}^2}$ could stand as a Nyquist rate for wavelet systems. An analysis of
the eigenvalue profile of the   wavelet Toeplitz operator corresponding to the Cauchy wavelet $T^{I}_{\psi_0^\alpha}=P_{\mathcal{W}_{\psi_0^\alpha}}\chi_I P_{\mathcal{W}_{\psi_0^\alpha}}$ made by Daubechies
and Paul \cite{dapa88} has shown that the analogies with the euclidean
problem solved by Landau were misleading, since the error resulting from the
difference between the first two moments of $T^{I}_{\psi_0^\alpha}$ (inherited from the
exponential growth of the `$ax+b$'-group) has precisely the same asymptotic
behavior as the trace itself, when the region $I\subset \C^+$ is dilated to approach the whole
upper half plane $\mathbb{C}^{+}$. The same thing happens for general
wavelets \cite{defeno02}. Nevertheless, Seip pursued the problem further in
a body of research that culminated in a complete description of sampling and
interpolation sequences in Bergman spaces \cite{Seip2}, an elegant
adaptation of a scheme of Beurling where several ideas from Korenblum's
`hyperbolic Nevanlinna theory' \cite{Kor} have been put in use and refined.
Unfortunately, the methods of \cite{Seip2} seem to fully depend on the
holomorphic structure of the space and, as recently proved \cite%
{AnalyticWavelet}, the only choice of $\psi $ lending itself to analytic
phase spaces is (up to a chirp) the Cauchy wavelet $\psi_0^\alpha (t)=(t+i)^{-\alpha
-1}$. After 30 years of research activity on wavelet frames (mostly
concentrated in the period 1990-2010), and despite the participation of a
significant group of mathematicians who introduced new ideas and found
several interesting partial results (we refer to \cite{HK,Kutt} for a sample
of this activity), a full solution of the problem (i.e., generic wavelets
and sampling sets) has been assigned a status between `impossible' and
`hopeless' (see \cite{Kutt} for a detailed explanation of the technical and
conceptual obstructions).

\subsection{von Neumann dimension}

Let us define the functions $\psi
_{n}^{\alpha }\in \H^2(\C^+),\ n\in \mathbb{N}_{0},$ via the Fourier transform   of generalized Laguerre polynomials
\begin{equation}
(\mathcal{F}\psi _{n}^{\alpha })(\xi ):=\xi ^{\frac{\alpha }{2}}e^{-\xi
}L_{n}^{\alpha }(2\xi ),\quad \xi >0,\quad \text{with}\quad L_{n}^{\alpha }(t)=\sum_{k=0}^n (-1)^k\binom{n+\alpha}{n-k}\frac{t^k}{k!}, \label{mother}
\end{equation}%
where we follow the convention
$$
\mathcal{F}f(\xi)=\int_\mathbb{R} f(t)e^{-i\xi t}dt.
$$
 In the present paper, we  show that, at least for the orbits $\Gamma (z),\ z\in\C^+,$
of a Fuchsian group $\Gamma$  and mother wavelets chosen from the family $\{\psi
_{n}^{\alpha }\}_{n\in \mathbb{N}_{0}}$, 
 there is indeed a Nyquist rate separating frames from Riesz sequences and sampling
from interpolating sets for all possible coherent states defined on the
phase space of $\mathcal{W}_{\psi _{n}^{\alpha }}$. And 
such a Nyquist rate is provided by the quantity $\frac{C_{\psi _{n}^{\alpha
}}}{\Vert \psi _{n}^{\alpha }\Vert _{2}^{2}}$, in perfect agreement with
Seip's 1989 conjecture \cite{SeipConjecture,SeipConj}. The technical reason
for the accuracy of this prediction is the following. Our main contribution
for the theory of von Neumann algebras is a dimension computation associated
with a projective unitary group representation of $PSL(2,\mathbb{R})$ on $%
L^{2}(\mathbb{C}^{+},\mu )$ for an infinite family of
non-analytic reproducing kernel Hilbert spaces indexed by $n$. This relies on Jones' new method \cite{Jones} which, according to the author, incorporates a suggestion of McMullen and is inspired by Atiyah's setting
for the covering space index theorem \cite{Atiyah} and boils down to the computation of the trace of the same operator that Seip considered  \cite%
{SeipConjecture,SeipConj}. Note that our approach
extends Radulescu's computation \cite{Radu} to general $n$. 

{This new approach seems to be more elementary than Bekka's \cite%
{Bekka} (for the analytic case $n=0$), where the computation of the von
Neumann dimension follows a scheme inspired by the derivation of
Atiyah-Schmid's formula \cite{AS} and also than Radulescu's \cite{Radu}. It
allows us to derive a Nyquist rate for phase spaces of the wavelet transform
(see Theorem~\ref{thm:1}) and for super wavelet systems (see Corollary~\ref{cor:super-wavelet}) using mother wavelets associated to the eigenspaces of the Maass operator.}

\subsection{Outline}

This paper is simply organized as follows: a section with the statements of
the main results and the important special cases of Bergman and Maass
spaces, and a `Background' section with the essential material on wavelets,
Fuchsian groups and von Neumann algebras is followed by a section with the proofs of
the main results.

\section{Results}

Let us
define the   projective unitary representation $\tau_n^\alpha$ of $PSL(2,\mathbb{R})$ on $L^2({%
\mathbb{C}}^+,\mu)$ by 
\begin{equation}  \label{eq:def-tau-in-results}
\tau_{n}^{\alpha }(\gamma)F(z):=\left(\frac{|cz+d|}{cz+d}\right)^{2n+%
\alpha+1}F(\gamma(z)), 
\end{equation}
where $d\mu (w)=\text{Im}(w)^{-2}dw$, $F\in L^2({\mathbb{C}}^+,\mu)$, $%
\gamma\in PSL(2,\mathbb{R})$, and $z\in {\mathbb{C}}^+$. Note that $\tau_{n}^{\alpha }$ differs from the classical representation of $PSL(2,\mathbb{R})$ as considered for example in \cite{Bekka,coxeter,Radu} by the phase factor $\big(\frac{|cz+d|}{cz+d}\big)^{2n+%
\alpha+1}$, and that, as we show later,  $\tau_n^\alpha$ leaves $\mathcal{W}_{\psi_n^\alpha}$ invariant.

For the definition of frames, Riesz sequences, uniqueness sets and sampling and interpolating sequences, we refer to Section~\ref{sec:frames}.

\begin{theorem}
\label{thm:1} Let $\Gamma \subset PSL(2,\mathbb{R})$ be a Fuchsian group
with fundamental domain $\Omega \subset \mathbb{C}^{+}$ (of finite volume)
and $F\in \mathcal{W}_{\psi_n^\alpha}$. If $\{\tau_{n}^{\alpha
}(\gamma)F\}_{\gamma \in \Gamma }$ is a frame for $\mathcal{W}_{\psi
_{n}^{\alpha }}$, then $\Omega $ is compact and 
\begin{equation}
\left\vert \Omega \right\vert \leq \frac{C_{\psi _{n}^{\alpha }}}{\Vert \psi
_{n}^{\alpha }\Vert _{2}^{2}}=\frac{2}{\alpha }\text{,}  \label{eq:thm-sam}
\end{equation}%
where $|\Omega |$ is calculated via the measure $\mu$.

If $\{\tau_{n}^{\alpha }(\gamma)F\}_{\gamma \in \Gamma }$ is a Riesz
sequence for $\mathcal{W}_{\psi _{n}^{\alpha }}$, then 
\begin{equation}
\left\vert \Omega \right\vert \geq \frac{C_{\psi _{n}^{\alpha }}}{\Vert \psi
_{n}^{\alpha }\Vert _{2}^{2}}=\frac{2}{\alpha }\text{.}  \label{eq:thm-int}
\end{equation}
\end{theorem}

\begin{remark}
Given a Hilbert space $\mathcal{H}$ and a notion of density of points, a
Nyquist rate can be formally defined as a precise value of density which
must be exceeded for the existence of frames for $\mathcal{H}$, and cannot
be exceeded for the existence of Riesz sequences for $\mathcal{H}$. There
are several notions of affine density (see those in \cite{Kutt} and
\cite{Seip2}), all of them essentially defined as the large $r$ limit of the
ratio between the number of points in a hyperbolic ball of radius $r$ and
the measure of the hyperbolic ball. For a  Fuchsian group $\Gamma $ one can
use the hyperbolic geodesic distance $d(z,w)$ with $z,w\in {\mathbb{C}}^{+}$
to define the counting function $n_{r}(z,w)=\#\{\gamma \in \Gamma
:d(w,\gamma (z))<r\}$ and the hyperbolic ball $B_{r}$. Then a result of
Patterson \cite{Patterson} gives%
\begin{equation*}
\lim_{r\rightarrow \infty }\frac{n_{r}(z,w)}{\left\vert B_{r}\right\vert }=%
\frac{1}{\left\vert \Omega \right\vert }\text{.}
\end{equation*}%
Thus, Theorem~\ref{thm:1} says that $\alpha /2$ is a Nyquist rate for\ all the Hilbert
spaces $\mathcal{W}_{\psi _{n}^{\alpha }}$, $n\in \mathbb{N}_{0}$.
\end{remark}

Applying the inverse wavelet transform to the function $\tau_n^\alpha(%
\gamma) F$, a Nyquist condition for a particular family of vectors in $\H^2({%
\mathbb{C}}^+)$ follows.

\begin{corollary}
\label{cor:other-family} Let $\Gamma \subset PSL(2,\mathbb{R})$ be a
Fuchsian group with fundamental domain $\Omega \subset \mathbb{C}^{+}$ (of
finite volume), $\Phi\in L^2({\mathbb{C}}^+,\mu)$, and the family of vectors 
$\{\varphi_\gamma\}_{\gamma\in\Gamma}$ be weakly defined by 
\begin{equation}  \label{eq:def-varphi}
\varphi_\gamma:=\int_{{\mathbb{C}}^+} \Phi(z) \left(\frac{cz+d}{|cz+d|}%
\right)^{2n+\alpha+1} \pi(\gamma(z))\psi_n^\alpha d\mu(z).
\end{equation}
If $\{\varphi_\gamma\}_{\gamma\in\Gamma}$ is a frame for $\H^2({\mathbb{C}}%
^+)$, then $\Omega$ is compact and \eqref{eq:thm-sam} holds, and if $%
\{\varphi_\gamma\}_{\gamma\in\Gamma}$ is a Riesz sequence for $\H^2({\mathbb{%
C}}^+)$, then \eqref{eq:thm-int} holds.
\end{corollary}

If we choose $F$ in Theorem~\ref{thm:1} to be  the reproducing kernel $k_{\psi _{n}^{\alpha
}}(z,\cdot )$ of the space $\mathcal{W}_{\psi_n^\alpha}$, then we can state the Nyquist condition for wavelet systems
generated by Fuchsian groups and wavelets $\psi _{n}^{\alpha }$. As recently observed
by Romero and Velthoven in a different formulation \cite{RV}, the following result
  can actually be obtained from the analytic case $n=0$ of Theorem~\ref{thm:1}
above, using Radulescu's extension of the dimension computation in the
projective case (non-integer $\alpha $) for $n=0$ \cite{Radu}. 

\begin{corollary}
\label{cor:wavelet} Let $\Gamma \subset PSL(2,\mathbb{R})$ be a Fuchsian
group with fundamental domain $\Omega \subset \mathbb{C}^{+}$ (of finite
volume), and $z\in \Omega $. If the family $\{\pi(\gamma(z))\psi_n^\alpha%
\}_{\gamma\in\Gamma}$ is a wavelet frame for $\H^2({\mathbb{C}}^+)$
(equivalently, if $\Gamma(z)$ is a sampling sequence for $\mathcal{W}_{\psi
_{n}^{\alpha }}$), then $\Omega$ is compact and \eqref{eq:thm-sam} holds.

Condition \eqref{eq:thm-sam} is also necessary if $\Gamma (z)$ is a
uniqueness set for the wavelet space $\mathcal{W}_{\psi _{n}^{\alpha }}$%
. Thus, if \eqref{eq:thm-int} holds, then %
there exists a function in the wavelet space $\mathcal{W}_{\psi _{n}^{\alpha
}}$ vanishing on $\Gamma (z)$.

If $\{\pi(\gamma(z))\psi_n^\alpha\}_{\gamma\in\Gamma}$ is a wavelet Riesz
sequence (equivalently, {if $\Gamma(z)$ is an} interpolating sequence for $%
\mathcal{W}_{\psi _{n}^{\alpha }} $), then \eqref{eq:thm-int} holds.
\end{corollary}

While $\Omega $ needs to be compact for the existence of frames, this is not
the case for uniqueness sets (see the discussion in \cite[Section 5]{KL}).

Observe that the above results leave $|\Omega |=C_{\psi _{n}^{\alpha
}}/\Vert \psi _{n}^{\alpha }\Vert _{2}^{2}$ as the only possibility for $%
\{\pi (\gamma (z))\psi _{n}^{\alpha }\}_{\gamma \in \Gamma }$ to be a
wavelet Riesz basis for $\mathcal{H}^{2}(\mathbb{C}^{+})$. In the case $n=0$%
, we are led to analytic Bergman spaces, which contain no sampling and
interpolation sequences: given $\Gamma $ interpolating, pick $f$ vanishing
on the orbit of $\Gamma -\{\gamma _{0}\}$ for some $\gamma _{0}\in \Gamma $
with $f(\gamma _{0}(z))=i$, and then turn $\Gamma $ into a zero sequence
setting $F(w)=\frac{w-i}{w+i}f(w)$; thus interpolating sequences cannot be
sampling. For other values of $n$, the spaces do not seem to be closed under
shift multipliers and nothing is known about the existence of such
sequences. The perturbation argument used in \cite{Seip} also does not seem
to work here since, at least for subgroups of $PSL(2,\mathbb{Z})$, a `small'
perturbation of the parameters would change the group structure.\ We leave
this as an interesting research question. So far, the only space known to
have sampling and interpolating sequences is the Paley-Wiener space, where
the multiplier algebra is trivial (see \cite{Seipbook} for an engrossing
discussion about this).

We   show that the choice (\ref{mother}) is, within some boundaries, the
only family leading to a well-defined problem. This can be seen as a
converse of the first step in the proof of the above results, where we show
that the specific wavelet spaces are invariant under the action of Fuchsian
groups which is not the case for general wavelets. In particular, we show
that these are the only wavelet spaces invariant under representations of
the form \eqref{eq:def-tau-in-results} if we assume reasonable restrictions
on $\psi$. 
To be precise, let $\mathcal{A}\subset \mathcal{H}^{2}(\mathbb{C}^{+})$ be
defined as 
\begin{equation}  \label{def:A}
\mathcal{A}:=\big\{f\in \mathcal{H}^{2}(\mathbb{C}^{+}):\ \widehat{f}(\xi
)\in \mathbb{R},\ \widehat{f}\widehat{f}\hspace{2pt} {}^{\prime },\xi 
\widehat{f}\widehat{f}\hspace{2pt} {}^{\prime \prime }\in L^{1}(\mathbb{R}%
^{+})\big\}\text{.}
\end{equation}%
The situation is depicted in the next result, which may be of independent
interest.

\begin{theorem}
\label{thm:3} If $F\in \mathcal{W}_{\psi _{n}^{\alpha }}$, then ${\tau }%
_{n}^{\alpha }(g)F\in \mathcal{W}_{\psi _{n}^{\alpha }}$ for every $g\in
PSL(2,\mathbb{R})$. Conversely, let $\sigma_{s}$ be a projective unitary
group representation of $PSL(2,\mathbb{R})$ on $L^{2}(\mathbb{C}^{+},\mu )$
of the form 
\begin{equation*}
\sigma_{s}(g)F(z)=\left( \frac{|cz+d|}{cz+d}\right) ^{s}F(g(z)).
\end{equation*}%
If $\psi \in \mathcal{A}$ and $\sigma_{s}$ leaves $\mathcal{W}_{\psi } $
invariant, then for some $n\in \mathbb{N}_{0}$ and $C\in \mathbb{C}%
\backslash \{0\}$ we have 
\begin{equation*}
\psi =C\psi _{n}^{\alpha },
\end{equation*}%
where $\alpha =s-2n-1$.
\end{theorem}

The first statement of Theorem~\ref{thm:3} is crucial for the proof of
Theorem~\ref{thm:1} and its corollaries as well as for Theorem~\ref{thm:super-wavelet}.

\begin{remark}
Theorem~\ref{thm:1} is an extension of Perelomov's result \cite{PerelomovCompleteness}  for general
$PSL(2,\R)$ coherent states in terms of $|\Omega|$, as suggested in \cite{PerelomovQuestion}, but it still lacks
a sufficient condition for a given $F\in \mathcal{W}_{\psi_n^\alpha}$
to be complete, since the classes of
incomplete and Riesz sequences may or may not be disjoint. However, the observations
in the proof of Theorem~\ref{thm:1} combined with the other implication of Theorem~\ref{thm-cyclic}
allow to state a sufficient condition in the following sense: if
$
|\Omega|\leq \frac{C_{\psi_n^\alpha}}{\|\psi_n^\alpha\|^2_2}=\frac{2}{\alpha}
$
then, for every $n\geq 0$, there exists $F\in\mathcal{W}_{\psi_n^\alpha}$, such that $\{\tau_n^\alpha(\gamma)F\}_{\gamma\in\Gamma}$ is complete for $\mathcal{W}_{\psi_n^\alpha}$.
\end{remark}


An interesting special case of the above result follows from the choice $%
\alpha =2(B-n)-1$ (this leads to a finite sequence of orthogonal spaces $%
\mathcal{W}_{\psi _{n}^{2(B-n)-1}}$, while the spaces $\mathcal{W}_{\psi _{n}^{\alpha }}$ are in
general not orthogonal) considered by Mouayn \cite{Mouhayn} to attach
coherent states to hyperbolic Landau levels. Let us consider the Maass
operator 
\begin{equation*}
H_{B}:=-(z-\overline{z})^{2}\frac{\partial ^{2}}{\partial z\partial 
\overline{z}}-B(z-\overline{z})\left( \frac{\partial }{\partial z}-\frac{%
\partial }{\partial \overline{z}}\right) =-s^{2}\left( \frac{\partial ^{2}}{%
\partial x^{2}}+\frac{\partial ^{2}}{\partial s^{2}}\right) +2iBs\frac{%
\partial }{\partial x}\text{.}
\end{equation*}%
The spectrum\ of $H_{B}$ consists of\textit{\ }a continuous part $\left[
1/4,+\infty \right[ $ and a finite number of eigenvalues. To each eigenvalue 
$\lambda _{n}^{B}=(B-n)\left( B-n-1\right) $, for $n=0,1,...,\lfloor B-\frac{%
1}{2}\rfloor $, corresponds an infinite dimensional reproducing kernel
Hilbert eigenspace $A_{n}^{B}\left( \mathbb{C}^{+},\mu \right) \subset L^{2}(%
{\mathbb{C}}^{+},\mu )$. By computing the reproducing kernels (see Section~%
\ref{sec:laguerre}) we can identify the spaces $\mathcal{W}_{\psi
_{n}^{2(B-n)-1}}
=A_{n}^{B}\left( \mathbb{C}^{+},\mu \right) $.\ From Corollary~\ref{cor:wavelet} it
then follows:

\begin{corollary}\label{cor:super-wavelet}
Let $n\leq \lfloor B-\frac{1}{2}\rfloor$. If $\Gamma (z)$ is a sampling 
sequence for the space $%
A_{n}^{B}\left( \mathbb{C}^{+},\mu \right) $, then $\Omega$ is compact and
\begin{equation*}
\left\vert \Omega \right\vert \leq \frac{1}{B-n-1/2}.
\end{equation*}%
If $\Gamma (z)$ is an interpolating sequence for the space $A_{n}^{B}\left( 
\mathbb{C}^{+},\mu \right) $, then%
\begin{equation*}
\left\vert \Omega \right\vert \geq \frac{1}{B-n-1/2}.
\end{equation*}
\end{corollary}

It follows from the result above that, if $\left\vert \Omega \right\vert >\frac{1}{B-n-1/2}$, then
one can find \ $F\in A_{n}^{B}\left( \mathbb{C}^{+},\mu \right) $\ vanishing
on $\Gamma (z)$. While in \cite{Jones} explicit examples of functions in $%
A_{0}^{B}\left( \mathbb{C}^{+},\mu \right) $ vanishing on $\Gamma (z)$\ are
provided by (analytic) automorphic forms, explicit examples of functions in $%
A_{n}^{B}\left( \mathbb{C}^{+},\mu \right) $ (which are only real analytic
and polyanalytic) are provided by the theory of Maass forms \cite%
{Maass,Roelcke}. Thus, it may be possible to use Maass cusp forms to adapt
the construction in \cite[Section 7]{Jones}, but we will not pursue this
topic in the present paper.

There is more to be said here. If we define
\begin{equation}\label{eq:def-oplus}
\textbf{A}_{N}^{B}\left( \mathbb{C}^{+},\mu \right) :=\bigoplus_{n=0}^{N-1} A_{n}^{B}\left( \mathbb{C}^{+},\mu \right)=\bigoplus_{n=0}^{N-1} \mathcal{W}_{\psi_n^{2(B-n)-1}},
\end{equation}
then there  also exist a Nyquist rate for $\textbf{A}_{N}^{B}\left( \mathbb{C}^{+},\mu \right)$ separating sampling and interpolation sequences:

\begin{theorem}\label{thm:super-wavelet}
Let $N\in\{1,\ldots, \lfloor B-\frac{1}{2}\rfloor\}$. If $\Gamma (z)$ is a sampling 
sequence for the space $%
\textbf{A}_{N}^{B}\left( \mathbb{C}^{+},\mu \right) $, then $\Omega$ is compact and
\begin{equation}\label{eq:super-nyq}
\left\vert \Omega \right\vert \leq \frac{2}{N(2B-N)}.
\end{equation}
If $\Gamma (z)$ is an interpolating sequence for the space $\textbf{A}_{N}^{B}\left( 
\mathbb{C}^{+},\mu \right) $, then%
\begin{equation}
\left\vert \Omega \right\vert \geq \frac{2}{N(2B-N)}.
\end{equation}
\end{theorem}

An equivalent interpretation of Theorem~\ref{thm:super-wavelet} (in the spirit of the second identity of \eqref{eq:def-oplus}) is that \eqref{eq:super-nyq} is a necessary condition for the vector valued system $\{\pi(\gamma(z))\boldsymbol{\phi}_N^B\}_{\gamma\in\Gamma}\subset \H^2(\C^+,\C^{N})$, where  $$\boldsymbol{\phi}_N^B:=\big(\psi_0^{2B-1},\psi_1^{2B-3},...,\psi_{N-1}^{2(B-N)+1}\big),$$ to be a \emph{wavelet superframe} for $\H^2(\C^+,\C^{N})$, see e.g.  \cite{ab17} for further information on super frames.  In particular, we get the following result.
\begin{corollary}
Let $N\in\{1,\ldots, \lfloor B-\frac{1}{2}\rfloor\}$. If $\{\pi(\gamma (z))\boldsymbol{\phi}_N^B\}_{\gamma\in\Gamma}$ is a wavelet superframe for $\H^2(\C^+,\C^N)$, then $\Omega$ is compact and
\begin{equation} 
\left\vert \Omega \right\vert \leq \frac{2}{N(2B-N)}.
\end{equation}
\end{corollary}
Note that  a similar necessary condition for super-wavelet frames generated by mother wavelets connected to the true polyanalytic Bergman  spaces and the hyperbolic lattice was shown in \cite{ab17}. We also refer to \cite{grlyu09} for a  result by Gr\"ochenig and Lyubarskii in the context of Gabor superframes with Hermite windows and of Balan \cite{Balan} providing necessary conditions for general Gabor superframes.

As a final remark, we outline the physical relevance of the results. The
Maass operator has been considered \cite{AoP,Comtet} as an alternative model
for the formation of Landau levels, an important model in the physics of the
Quantum-Hall effect. Here, a constant magnetic field of intensity $B$ acts
on the open hyperbolic plane realized as the Poincar\'{e} upper half-plane $%
\mathbb{C}^{+}$, leading to the formation of a mixed spectrum with a {%
discrete part} corresponding to \textit{bound states }(\emph{hyperbolic
Landau levels}), and a continuous part corresponding to \textit{scattering
states}. The finite number of the levels, which depends on the strength of
the magnetic field is, in principle, a physical advantage of such a model. In 
\cite{AoP} the first bounds on the size of the fundamental region, which
correspond to a `cell' (Pauli exclusion principle allows only one electron
per cell), have been obtained. Note that the above improves the main result of \cite%
{AoP} by a factor of $n+1$. Given a sample modelled by a bounded region $%
\Delta \subset \mathbb{C}^{+}$ in a hyperbolic Landau level, assuming the
Pauli exclusion principle, Corollary~\ref{cor:wavelet} provides estimates for the number
of particles distributed by the model on such a region.

\section{Background}

\subsection{Wavelets on the Hardy space}

\label{sec:wavelets}

The affine group $G_{a}=\mathbb{R}\times \mathbb{R}^{+}$ is defined by the
group law 
\begin{equation*}
(b,a)\cdot (x,s)=(ax+b,as),\quad (b,a),(x,s)\in G_{a}.
\end{equation*}%
Here, we   identify $G_{a}$ with the upper half plane $\mathbb{C}^{+}$
and note that the left-invariant Haar measure on the affine group is given
by 
\begin{equation*}
d\mu (z):=\frac{dxdy}{s^{2}},\quad \text{ where }z=x+is\in \mathbb{C}^{+}.
\end{equation*}%
Let us define a unitary representation of $G_{a}$ on the Hardy space $%
\mathcal{H}^{2}(\mathbb{C}^{+})$ by 
\begin{equation*}
\pi (z)\psi (t):=T_{x}D_{s}\psi (t)=\frac{1}{\sqrt{s}}\psi \left( \frac{t-x}{%
s}\right) ,\quad z=x+is\in \mathbb{C}^{+}.
\end{equation*}%
The \emph{wavelet transform} of a function $f\in \mathcal{H}^{2}(\mathbb{C}%
^{+})$ using the \emph{mother wavelet} $\psi \in \mathcal{H}^{2}(\mathbb{C}%
^{+})$ is then defined as 
\begin{equation}\label{eq:def-wt}
W_{\psi }f(z):=\langle f,\pi (z)\psi \rangle =\int_{\mathbb{R}}f(t)\overline{%
T_{x}D_{s}\psi (t)}dt=\sqrt{s}\int_{0}^{\infty }\widehat{f}(\xi)\overline{%
\widehat{g}\left(s \xi \right) }e^{ix\xi}{d\xi},
\end{equation}%
where $\psi $ satisfies the \emph{admissibility condition} 
\begin{equation*}
C_{\psi }:=\int_{\mathbb{R}^{+}}|\widehat{\psi }(\omega )|^{2}\frac{d\omega 
}{|\omega |}<\infty .
\end{equation*}%
%
%
For $f_{1},f_{2}\in \mathcal{H}^{2}(\mathbb{C}^{+})$ and $\psi $ admissible,
the following orthogonality relation holds 
\begin{equation}
\int_{\mathbb{C}^{+}}W_{\psi }f_{1}(z)\overline{W_{\psi }f_{2}(z)}d\mu
(z)=C_{\psi }\langle f_{1},f_{2}\rangle .  \label{eq:orth-rel}
\end{equation}%
In particular, if we set $f_{2}=\pi (z)\psi $, then for every $f\in \mathcal{%
H}^{2}(\mathbb{C}^{+})$ one has 
\begin{equation}
W_{\psi }f(z)=\frac{1}{C_{\psi }}\int_{\mathbb{C}^{+}}W_{\psi }f(w)\langle
\pi (w)\psi ,\pi (z)\psi \rangle d\mu (z),\quad z\in \mathbb{C}^{+}.
\label{eq:rep-eq}
\end{equation}%
This reproducing formula then implies that the range of the wavelet
transform 
\begin{equation*}
\mathcal{W}_{\psi }:=\big\{F\in L^{2}(\mathbb{C}^{+},\mu ):\ F=W_{\psi }f,\
f\in \mathcal{H}^{2}(\mathbb{C}^{+})\big\}
\end{equation*}%
is a reproducing kernel subspace of $L^{2}(\mathbb{C}^{+},\mu )$ with kernel 
\begin{equation*}
k_{\psi }(z,w)=\frac{1}{C_{\psi }}\langle \pi (w)\psi ,\pi (z)\psi \rangle 
,\quad \text{and}\quad
k_{\psi }(z,z)=\frac{\Vert \psi \Vert _{2}^{2}}{C_{\psi }}.
\end{equation*}%
Moreover, $k_\psi$ is the integral kernel of the projection operator $P_\psi:L^2(\C^+,\mu)\rightarrow \mathcal{W}_\psi$.
The reproducing kernel of $\mathcal{W}_\psi$ can easily be computed using the Fourier transform $%
\mathcal{F}:\mathcal{H}^{2}(\mathbb{C}^{+})\rightarrow L^{2}(0,\infty )$ as
follows%
\begin{align}
\begin{split}
\left\langle \pi(w)g,\pi ( z)g\right\rangle &=\left\langle \mathcal{F}%
(\pi(w)g),\mathcal{F}(\pi ( z)g)\right\rangle _{L^{2}(0,\infty )} \\
&=\left(s s^{\prime }\right)^{\frac{1}{2}}\int_{0}^{\infty }\widehat{g}%
(s^\prime \xi )\overline{\widehat{g}\left(s\xi \right)}e^{i(x-x^{\prime
})\xi }d\xi \text{.}  \label{Kg2}
\end{split}%
\end{align}

\subsection{Wavelet frames and Riesz sequences}\label{sec:frames}

Let $\H$ be a separable Hilbert space and $\mathcal{I}$ be a discrete index set. The family of vectors  $\{\psi_k\}_{k\in \mathcal{I}}$ is called a \emph{frame}
 if there exist
constants $A,B>0$ such that, for every $f\in \mathcal{H}$, one has 
\begin{equation}
A\Vert f\Vert ^{2}\leq \sum_{k\in\mathcal{I}}|\left\langle f,\psi_k \right\rangle |^{2}\leq B\Vert f\Vert ^{2}.  \label{WaveletFrame}
\end{equation}
If, on the other hand,  there
exist constants $A,B>0$ such that, for every $c\in \ell ^{2}(\mathcal{I})$
\begin{equation}
A\Vert c\Vert _{\ell ^{2}(\mathcal{I} )}^{2}\leq \Big\|\sum_{k\in\mathcal{I}
}c_{k}\psi_k \Big\|^{2}\leq B\Vert c\Vert _{\ell
^{2}(\mathcal{I} )}^{2},  \label{WaveletRiesz}
\end{equation}%
then $\{\psi_k\}_{k\in \mathcal{I}}$ is called a \emph{Riesz sequence}. A Riesz sequence that is also a frame is called a \emph{Riesz basis}.

Let $\Gamma \subset \mathbb{C}^{+}$ be a discrete set and $\psi \in \mathcal{%
H}^{2}(\mathbb{C}^{+})$ be admissible. We call $\{\pi (\gamma )\psi
\}_{\gamma \in \Gamma }$   a \emph{wavelet frame} (resp. \emph{wavelet Riesz sequence}) if it is a frame (resp. Riesz sequence) for $\H^2(\C^+)$.
Using (\ref{eq:rep-eq}) we can observe that $\{\pi (\gamma )\psi
\}_{\gamma \in \Gamma }$ is a frame for $\H^2(\C^+)$ if and only if 
  $\left\{ \overline{k_{\psi }(\gamma ,\cdot )}\right\} _{\gamma \in
\Gamma }$\ is a frame for $\mathcal{W}_{\psi }$. In particular, $\left\{ 
\overline{k_{\psi }(\gamma ,\cdot )}\right\} _{\gamma \in \Gamma }$ is \emph{%
dense} in $\mathcal{W}_{\psi }$.  

We say that $\Gamma $ is a \emph{sampling sequence} for the wavelet space $%
\mathcal{W}_{\psi }$ if there exist constants $A,B>0$ such that for every $%
F\in \mathcal{W}_{\psi }$ 
\begin{equation*}
A\Vert F\Vert ^{2}\leq \sum_{\gamma \in \Gamma }{|F(\gamma )|^{2}}\leq
B\Vert F\Vert ^{2},
\end{equation*}%
and $\Gamma $ is called a \emph{uniqueness set} for $\mathcal{W}_{\psi }$ if $F(\gamma)=0,$ for every $\gamma\in\Gamma$ implies $F=0$.
Likewise, we say that $\Gamma $ is an \emph{interpolating sequence} if $\left\{ 
\overline{k_{\psi }(\gamma ,\cdot )}
\right\} _{\gamma \in \Gamma }$\ is a Riesz sequence for $\mathcal{W}_{\psi }
$.

For $\boldsymbol{f}=(f_0,f_1,...,f_{N-1}),\boldsymbol{\psi}=(\psi_0,\psi_1,...,\psi_{N-1})\in \H^2(\C^+,\C^{N})$ we define  
$$
\boldsymbol{W}_{\boldsymbol{\psi}}\boldsymbol{f}(z):=\sum_{n=0}^{N-1} W_{\psi}f_n(z),\quad z\in\C^+,
$$
and say that $\{\pi(\gamma(z))\boldsymbol{\psi}\}_{\gamma\in\Gamma}$ is a \emph{wavelet superframe} if there exist $A,B>0$ such that for every $\boldsymbol{f}\in\H^2(\C^+,\C^N)$
$$
A\|\boldsymbol{f}\|_{\H^2(\C^+,\C^{N})}^2\leq \sum_{\gamma\in\Gamma}|\boldsymbol{W}_{\boldsymbol{\psi}}\boldsymbol{f}(\gamma)|^2\leq B\|\boldsymbol{f}\|_{\H^2(\C^+,\C^{N})}^2.
$$

\subsection{The Laguerre functions}

\label{sec:laguerre} Throughout this contribution we are interested in
analyzing wavelets $\psi _{n}^{\alpha }$ which are defined via their Fourier
transform in terms of the \emph{generalized Laguerre polynomials} $%
L_{n}^{\alpha }$ as 
\begin{equation*}
(\mathcal{F}\psi _{n}^{\alpha })(\xi ):=\xi ^{\frac{\alpha }{2}}e^{-\xi
}L_{n}^{\alpha }(2\xi ),\quad \text{with}\quad L_{n}^{\alpha }(\xi
)=\sum_{k=0}^{n}(-1)^{k}\binom{n+\alpha }{n-k}\frac{\xi ^{k}}{k!}.
\end{equation*}%
The orthogonality relation for generalized Laguerre polynomials reads 
\begin{equation*}
\int_{0}^{\infty }t^{\alpha }e^{-t}L_{n}^{\alpha }(t)L_{m}^{\alpha }(t)dt=%
\frac{\Gamma (n+\alpha +1)}{n!}\delta _{n,m}.
\end{equation*}%
From \cite[Equation (4)]{srimaal03}, it follows that 
\begin{equation}
\int_{0}^{\infty }t^{\alpha -1}e^{-t}L_{n}^{\alpha }(t)^{2}dt=\frac{\Gamma
(n+\alpha +1)}{n!\alpha}.  \label{eq:formula-norm}
\end{equation}

\noindent Consequently, 
\begin{equation}
\frac{C_{\psi _{n}^{\alpha }}}{\Vert \psi _{n}^{\alpha }\Vert _{2}^{2}}=%
\frac{2 }{\alpha} .  \label{eq:frac-c-norm}
\end{equation}%
For $z=x+is,w=y+iv\in {\mathbb{C}}^{+}$, the reproducing kernel of $\mathcal{%
W}_{\psi _{n}^{\alpha }}$ is given by 
\begin{equation}
k_{\psi _{n}^{\alpha }}(z,w)=2^\alpha \alpha\left( \frac{\overline{z}-w}{z-%
\overline{w}}\right) ^{n}\left( \frac{\sqrt{sv}}{-i(z-\overline{w})}\right)
^{\alpha +1}P_{n}^{(\alpha ,0)}\left( 1-\frac{8sv}{|z-\overline{w}|^{2}}%
\right) ,  \label{eq:repr-kernel}
\end{equation}%
where $P_{n}^{(\alpha ,\beta )}$ denotes the Jacobi polynomial 
\begin{equation*}
P_{n}^{(\alpha ,\beta )}(t)=\frac{\Gamma (\alpha +n+1)}{n!\Gamma (\alpha
+\beta +n+1)}\sum_{k=0}^{n}\binom{n}{k}\frac{\Gamma (\alpha +\beta +n+k+1)}{%
\Gamma (\alpha +k+1)}\left( \frac{t-1}{2}\right) ^{k}.
\end{equation*}%
Note that \eqref{eq:repr-kernel} can be derived using \eqref{Kg2} and \cite[%
p. 809, 7.414 (4)]{Grad}.

\subsection{A particular representation of \emph{PSL}($\pmb{2,\mathbb{R}}$)
and Fuchsian groups}

\label{sec:fuchs}


Let $I$ be the identity matrix in $\mathbb{R}^{2}$. The quotient group 
\begin{equation*}
PSL (2,\mathbb{R}):=SL(2,\mathbb{R})/\{\pm I\}
\end{equation*}%
is also known as the group of \emph{Moebius transformations} and can be
identified with the \emph{Poincar\'{e} half plane} $\mathbb{C}^{+}$. The element $g=%
\left(\begin{smallmatrix}
a & b \\ 
c & d%
\end{smallmatrix} \right)%
\in PSL (2,\mathbb{R})$ acts on $\mathbb{C}^{+}$ as 
\begin{equation*}
g(z):=\frac{az+b}{cz+d}.
\end{equation*}%
Note that the measure $\mu$ is invariant under the action of $g\in PSL(2,%
\mathbb{R})$. Let us define $j:PSL (2,\mathbb{R})\times\mathbb{C}%
^+\rightarrow \mathbb{C}\backslash\{0\}$ by 
\begin{equation*}
j(g,z)=(cz+d)^{-1}.
\end{equation*}
It is easy to check that $j$ satisfies the \emph{cocycle condition} 
\begin{equation*}
j(gh,z)=j(g,h(z))j(h,z).
\end{equation*}
Hence, if we use the principal branch of the argument (i.e. $\text{arg}%
(z)\in (-\pi,\pi]$) to define $z^\nu$, we have
\begin{equation*}
j(gh,z)^\nu=\lambda(g,h,\nu)j(g,h(z))^\nu j(h,z)^\nu,
\end{equation*}
for a unimodular function $\lambda$. It then immediately follows that $%
 \tau_n^\alpha:PSL (2,\mathbb{R})\to \mathcal{U}(L^2({\mathbb{C}}%
^+,\mu))$ given by 
\begin{equation}  \label{eq:tau-n}
 {\tau}_n^\alpha(g)F(z):= \left(\frac{|cz+d|}{cz+d}%
\right)^{2n+\alpha+1}F(g(z)),
\end{equation}%
defines a projective unitary group representation of $PSL (2,\mathbb{R})$ on 
$L^{2}(\mathbb{C}^{+},\mu )$. Note that $\tau_n^\alpha$ reduces to the standard representation in the analytic case ($n=0$).

Following Jones \cite[Sect. 4]{Jones} (the finite area condition is not
included by all authors), we define a \emph{Fuchsian group} $\Gamma $ as a  {%
discrete subgroup of }$PSL (2,\mathbb{R})$ {\ for which there exists a
fundamental domain with finite hyperbolic area (finite co-volume)}.

A \emph{fundamental domain} $\Omega$ for $\Gamma $ is a closed subset of $%
\mathbb{C}^{+}$ such that the images of $\Omega$ under $\Gamma $ cover $%
\mathbb{C}^+$, $\Omega$ is the closure of its interior $\Omega^{0}$, and for
each two distinct $\gamma_1,\gamma_2\in\Gamma$ one has $\gamma_1 (\Omega)^0
\cap \gamma_2 (\Omega)^0=\emptyset$.

\subsection{Von Neumann Algebras}

\label{sec:vna}

In this section, we closely follow the brief exposition of the topic in \cite%
{Jones} and repeat the results that are necessary for our arguments. For a thorough
introduction to von Neumann algebras we refer for example to \cite{jon15}.

A \emph{von Neumann algebra} $M$ is a $\ast$-closed unital algebra of bounded
operators on a (complex) Hilbert space $\mathcal{H}$ which is closed under
the topology of pointwise convergence on $\mathcal{H}$. A vector $\psi\in 
\mathcal{H}$ is called \emph{cyclic} if $M\psi$ is dense in $\mathcal{H}$
and $\psi$ is called \emph{separating} if $T\mapsto T\psi,\ T\in M$, is
injective. 
If there exists a trace functional  $\text{tr}:M\rightarrow\C$, such that $\text{tr}(AB)=\text{tr}(BA)$, $A,B\in M$, and $\text{tr}(I)=1$, then $M$ is called \emph{finite} and  we define the space $L^2(M)$ as the completion of $M$
with respect to the pre-Hilbert space inner product $\langle A,B\rangle =%
\text{tr}(B^\ast A)$.

It is a standard result on von Neumann algebras that if $\mathcal{H}$ is any
Hilbert space on which $M$ acts, then there is an $M$-linear isometry%
\begin{equation*}
U:\mathcal{H\rightarrow }\bigoplus\limits_{n=1}^{\infty }L^{2}(M)\text{.}
\end{equation*}
On $\bigoplus\limits_{n=1}^{\infty }L^{2}(M)$, the commutant $M'$ of $M$ admits a  trace $\text{Tr}_{L^2}$ which is canonically normalized such that the trace of any projection onto one of the $L^2(M)$'s is equal to $1$.
With that convention, we define the \emph{von Neumann dimension} by%
\begin{equation*}
\dim _{M}\mathcal{H}=\text{Tr}_{L^{2} }(UU^{\ast })\text{.}
\end{equation*}
The following elementary results on the von Neumann dimension will be the
key component in our arguments.

\begin{theorem}
\label{thm-cyclic} There is a cyclic vector for $M$ if and only if $\text{dim%
}_{M}\mathcal{H}\leq 1$.
\end{theorem}

\begin{theorem}
\label{thm-separating} There is a separating vector for $M$ if and only if $%
\text{dim}_{M}\mathcal{H}\geq 1$.
\end{theorem}

\begin{definition} 
Let $\Gamma$ be a (countable) discrete group. The \emph{von Neumann algebra of} $\Gamma$, which we denote by $vN(\Gamma)$, is the von Neumann algebra on $\ell^2(\Gamma)$ generated by the left regular representation $\gamma\mapsto \lambda_\gamma$, where $\lambda_{\gamma}(f)(\gamma')=f(\gamma^{-1}\gamma')$.

More generally, if $\omega:\Gamma\times\Gamma\rightarrow\mathbb{T}$ is a unit circle valued 2-cocycle, $vN_\omega(\Gamma)$ is generated on $\ell^2(\Gamma)$ by the unitaries $\lambda_\gamma^\omega$ where $\lambda_{\gamma}^\omega f(\gamma')=\omega(\gamma,\gamma')f(\gamma^{-1}\gamma')$.
\end{definition}

\noindent The subsequent proposition allows us to explicitly calculate the
von Neumann dimension in the particular case that the von Neumann algebra is
generated by a group representation of a discrete, infinite conjugacy class
(ICC) group. The results can be found in \cite[Proposition~3.6 \& Corollary~3.7]{Jones}, but similar calculations have already appeared in the proof of \cite[Theorem~3.3.2]{coxeter}.

\begin{proposition}
\label{dim}Let $\Gamma $ be a discrete ICC group and $\gamma \mapsto \tau
(\gamma )$ a projective group representation on $\mathcal{H}$ with $2$%
-cocycle $\omega $. Suppose there is a projection $Q$ on $\mathcal{H}$ such
that 
\begin{equation*}
\tau (\gamma ^{-1})Q\tau (\gamma )\perp Q\text{, \ \ \ \ for }\gamma \in
\Gamma -\{e\}\text{, \ \ \ and \ \ \ }\sum_{\gamma \in \Gamma }\tau (\gamma
^{-1})Q\tau (\gamma )=I\text{.}
\end{equation*}%
Then the following statements hold:
\begin{enumerate}
\item[$(i)$] There is a $\Gamma$-linear unitary map $U:\H\rightarrow \ell^2(\Gamma)\otimes Q\H$  with $U\tau(\gamma)U^{-1}=\lambda_\gamma^\omega\otimes Id$, for $\gamma\in\Gamma$.
\item[$(ii)$] The action of $\Gamma$ on $\H$ makes it into a $vN_\omega(\Gamma)$-module.

\item[$(iii)$] If $P$ is another projection on $\mathcal{H}$ that commutes with $\tau
(\gamma )$ for every $\gamma \in \Gamma $, then%
\begin{equation*}
\dim _{vN_{\omega }(\Gamma )}P\mathcal{H}=\emph{trace}_{B(\mathcal{H})}(QPQ)%
\text{,}
\end{equation*}
where $\emph{trace}_{B(\mathcal{H})}$ denotes the usual trace.
\end{enumerate}
\end{proposition}

\noindent Note that $(iii)$ is not explicitly formulated in this form
in \cite[Corollary~3.7]{Jones} but it appears in the last line of the proof.

\section{Proof of the main results}

\label{sec:proof}

\subsection{Proof of Theorem~\protect\ref{thm:3}}

First, we observe that the wavelet space $\mathcal{W}_{\psi _{n}^{\alpha }}$
is invariant under the action of ${\tau }_{n}^{\alpha }$ , the
projective unitary group representation of $PSL (2,\mathbb{R})$ on $L^{2}(%
\mathbb{C}^{+},\mu )$ defined in Section \ref{sec:fuchs}.

\begin{lemma}\label{lem:invariance}
If $F\in \mathcal{W}_{\psi_n^\alpha}$, then $  \tau_n^\alpha(g)F\in 
\mathcal{W}_{\psi_n^\alpha}$ for every $g\in PSL(2,\mathbb{R})$.
\end{lemma}

\proof Let $F\in \mathcal{W}_{\psi_n^\alpha}$. As $\mathcal{W}_{\psi _{n}^{\alpha }}$ is a reproducing kernel
Hilbert space, it follows that $ {\tau }_{n}^{\alpha }(\gamma )F\in 
\mathcal{W}_{\psi _{n}^{\alpha }}$ if and only if for every $z\in {\mathbb{C}%
}^{+}$ 
\begin{equation*}
 {\tau }_{n}^{\alpha }(g)F(z)=\int_{C^{+}} \tau _{n}^{\alpha
}(g)F(w)k_{\psi _{n}^{\alpha }}(z,w)d\mu (w).
\end{equation*}%
Let us now establish an identity for $k_{\psi _{n}^{\alpha }}$ that allows
us to verify the previous condition. Note that 
\begin{equation*}
\text{Im}(g(z))=\frac{\text{Im}(z)}{|cz+d|^{2}},\quad \text{and}\quad g^{-1}=%
\begin{pmatrix}
d & -b \\ 
-c & a%
\end{pmatrix}%
.
\end{equation*}%
It is then straightforward to show the following relations 
\begin{align}
\frac{\text{Im}(g(z))\text{Im}(w)}{|g(z)-\overline{w}|^{2}}& =\frac{\text{Im}%
(z)\text{Im}\left( g^{-1}(w)\right) }{\left\vert z-\overline{g^{-1}(w)}%
\right\vert ^{2}},  \label{eq:identities-1} \\
\frac{|cz+d|}{cz+d}\ \frac{\sqrt{\text{Im}(g(z))\text{Im}(w)}}{-i\big(g(z)-%
\overline{w}\big)}& =\frac{-cw+a}{|-cw+a|}\ \frac{\sqrt{\text{Im}(z)\text{Im}%
(g^{-1}(w))}}{-i\big(z-\overline{g^{-1}(w)}\big)}, \\
\left( \frac{|cz+d|}{cz+d}\right) ^{2}\ \frac{\overline{g(z)}-w}{g(z)-%
\overline{w}}& =\left( \frac{-cw+a}{|-cw+a|}\right) ^{2}\ \frac{\overline{z}%
-g^{-1}(w)}{z-\overline{g^{-1}(w)}}.  \label{eq:identities-3}
\end{align}%
Plugging equations \eqref{eq:identities-1}-\eqref{eq:identities-3} into the
explicit expression for $k_{\psi _{n}^{\alpha }}$ from \eqref{eq:repr-kernel}
yields 
\begin{equation}
\left( \frac{|cz+d|}{cz+d}\right) ^{2n+\alpha +1}k_{\psi _{n}^{\alpha
}}(g(z),w)=\left( \frac{-c{w}+a}{|-c{w}+a|}\right) ^{2n+\alpha +1}k_{\psi
_{n}^{\alpha }}(z,g^{-1}(w)).  \label{eq:identity-k-psi}
\end{equation}%
Since we assumed that $F\in \mathcal{W}_{\psi _{n}^{\alpha }}$, this shows 
\begin{align}
\begin{split}
 \tau _{n}^{\alpha }(g)F(z)& =\left( \frac{|cz+d|}{cz+d}\right)
^{2n+\alpha +1}F(g(z)) \\
& =\left( \frac{|cz+d|}{cz+d}\right) ^{2n+\alpha +1}\int_{\mathbb{C}%
^{+}}F(w)k_{\psi _{n}^{\alpha }}(g(z),w)d\mu (w) \\
& =\int_{\mathbb{C}^{+}}F(w)\left( \frac{-cw+a}{|-cw+a|}\right) ^{2n+\alpha
+1}k_{\psi _{n}^{\alpha }}(z,g^{-1}(w))d\mu (w) \\
& =\int_{\mathbb{C}^{+}}F(g(w))\left( \frac{-cg(w)+a}{|-cg(w)+a|}\right)
^{2n+\alpha +1}k_{\psi _{n}^{\alpha }}(z,w)d\mu (w) \\
& =\int_{\mathbb{C}^{+}}F(g(w))\left( \frac{|c w+d|}{cw+d}\right)
^{2n+\alpha +1}k_{\psi _{n}^{\alpha }}(z,w)d\mu (w) \\
& =\int_{\mathbb{C}^{+}} \tau _{n}^{\alpha }(g)F(w)k_{\psi
_{n}^{\alpha }}(z,w)d\mu (w),  \label{eq:p-commutes}
\end{split}%
\end{align}%
where we used the cocycle relation $j(g^{-1},g(w))=j(g,w)^{-1}$  
to derive the second to last identity. \hfill $\Box $\newline

We now prove the second statement of Theorem~\ref{thm:3}, which
characterizes the functions invariant under the kind of projective unitary
group representation of $PSL(2,\mathbb{R})$ on $L^{2}(\mathbb{C}^{+},\mu )$
we are using.

\begin{lemma}
Let $\mathcal{A}$ be given by \eqref{def:A} and $\sigma _{s}$ be a projective
unitary group representation of $PSL (2,\mathbb{R})$ on $L^{2}(\mathbb{C}%
^{+},\mu )$ of the form 
\begin{equation*}
\sigma _{s}(g)F(z)=\left( \frac{|cz+d|}{cz+d}\right) ^{s}F(g(z)).
\end{equation*}%
If $\psi \in \mathcal{A}$ and $\sigma _{s}$ leaves $\mathcal{W}_{\psi } $ invariant, then for some $n\in \mathbb{N}_{0}$
and $C\in \mathbb{C}\backslash \{0\}$ we have 
\begin{equation*}
\psi =C\psi _{n}^{\alpha },
\end{equation*}%
where $\alpha =s-2n-1$.
\end{lemma}

\proof If $\sigma _{s}$ leaves $\mathcal{W}_{\psi }$ invariant, then 
\begin{eqnarray*}
\sigma _{s}(g) F(z) &=&\int_{\mathbb{C}^{+}}\left( \frac{|cw+d|}{cw+d}\right)
^{s}F (g(w))k_{\psi }(z,w)d\mu (w) \\
&=&\int_{\mathbb{C}^{+}}F(w)\left( \frac{-cw+a}{|-cw+a|}\right) ^{s}k_{\psi
}(z,g^{-1}(w))d\mu (w).
\end{eqnarray*}%
On the other hand, one has
\begin{equation*}
\sigma _{s}(g) F(z)=\int_{\mathbb{C}^{+}}F(w)\left( \frac{|cz+d|}{cz+d}\right)
^{s}k_{\psi }(g(z),w)d\mu (w)\text{.}
\end{equation*}%
Combining the two preceding identities, it follows that%
\begin{equation}  \label{invariance}
\left( \frac{|cz+d|}{cz+d}\right) ^{s}k_{\psi }(g(z),w)=\left( \frac{-cw+a}{%
|-cw+a|}\right) ^{s}k_{\psi }(z,g^{-1}(w)),
\end{equation}%
for every $g\in PSL (2,\mathbb{R})$ and every $z,w\in \mathbb{C}^{+}$.
Consequently, setting $g=g_{\theta }=%
\begin{pmatrix}
\cos \theta & -\sin \theta \\ 
\sin \theta & \cos \theta%
\end{pmatrix}%
$, $\theta \in \lbrack 0,2\pi )$, and $w=i$ yields the necessary condition 
\begin{equation*}
\left( \frac{|z\sin \theta +\cos \theta |}{z\sin \theta +\cos \theta }%
\right) ^{s}W_{\psi }\psi (g_{\theta }(z))=e^{i\theta s}W_{\psi }\psi (z).
\end{equation*}%
Now, multiply both sides by $e^{i\theta s}\text{Im}(z)^{-s/2}$ to show
\begin{equation*}
\left( \frac{e^{i\theta }}{z\sin \theta +\cos \theta }\right) ^{s}W_{\psi
}^{s}\psi (g_{\theta }(z))=W_{\psi }^{s}\psi (z),
\end{equation*}%
where $W_{\psi }^{s}f(z):=\text{Im}(z)^{-s/2}W_{\psi }f(z)$. Thus, $W_{\psi
}^{s}f(z)$ is $SO(2)$-invariant in the sense of \cite{Rotation}. From \cite[Theorem~2.1]%
{Rotation} it then follows that $\psi $ is given in the Fourier domain by 
\begin{equation*}
\widehat{\psi }(\xi )=C\xi ^{\frac{s-1-2n}{2}}e^{-\xi }L_{n}^{s-1-2n}(2\xi ),\quad \xi
>0,
\end{equation*}%
for some $C\in \mathbb{C}\backslash \{0\}$, and $n\in \mathbb{N}_{0}$ with $%
n<\frac{s-1}{2}$. If we set $\alpha  =s-1-2n$ and solve for $s$, then we get 
$\psi =C\psi _{n}^{\alpha }$ together with $\sigma_s=  {\tau}%
_{n}^{\alpha }$. \hfill$\Box$ 

\subsection{Some preparatory results}
Let now $\Gamma $ be a Fuchsian group with fundamental domain $\Omega\subset 
\mathbb{C}^{+}$, and $\psi _{n}^{\alpha }$ as defined in Section~\ref%
{sec:laguerre}. Note that every Fuchsian group is ICC \cite{ake81}.

With a slight abuse of notation we   now consider $\tau_n^\alpha :\Gamma \rightarrow \mathcal{U}(\mathcal{W}%
_{\psi_n^\alpha })$ to be the restriction of $ {\tau}_n^\alpha$ to $%
\Gamma $ and $\mathcal{W}_{\psi_n^\alpha }$ which is a projective
representation of $\Gamma$.

In the next Lemma, we verify the conditions of Proposition~\ref{dim}.

\begin{lemma}
\label{lem:proj-props} Let $\Gamma \subset PSL (2,\mathbb{R})$ be a Fuchsian
group with fundamental domain $\Omega \subset \mathbb{C}^{+}$, $P_{\mathcal{W%
}_{\psi _{n}^{\alpha }}}:L^{2}(\mathbb{C}^{+},\mu )\rightarrow \mathcal{W}%
_{\psi _{n}^{\alpha }}$ be the projection onto $\mathcal{W}_{\psi
_{n}^{\alpha }}$, and 
\begin{equation*}
Q_{\Omega }:L^{2}(\mathbb{C}^{+},\mu )\rightarrow L^{2}(\mathbb{C}^{+},\mu
),\quad Q_{\Omega }f(z):=\chi _{\Omega }(z)f(z),
\end{equation*}%
where $\chi_\Omega$ denotes the characteristic function of $\Omega$.
Then $\tau _{n}^{\alpha }(\gamma )Q_{\Omega }\tau _{n}^{\alpha }(\gamma
^{-1})\bot Q_{\Omega }$, for every $\gamma \in \Gamma -\{e\}$, 
\begin{equation*}
\sum_{\gamma \in \Gamma }\tau _{n}^{\alpha }(\gamma )Q_{\Omega }\tau
_{n}^{\alpha }(\gamma ^{-1})=I_{L^{2}(\mathbb{C}^{+},\mu )}.
\end{equation*}%
and $P_{\mathcal{W}_{\psi _{n}^{\alpha }}}$ commutes with $\tau _{n}^{\alpha
}(\gamma )$ for every $\gamma \in \Gamma $.
\end{lemma}

\proof First, we have that $\tau _{n}^{\alpha }(\gamma )Q_{\Omega }\tau
_{n}^{\alpha }(\gamma ^{-1})=Q_{\gamma ^{-1}(\Omega )}\bot Q_{\Omega }$, for 
$\gamma \neq e$, since 
\begin{align*}
\tau _{n}^{\alpha }(\gamma )Q_{\Omega }\tau _{n}^{\alpha }(\gamma
^{-1})F(z)& =\tau _{n}^{\alpha }(\gamma )\left( \chi _{\Omega }(z)\cdot
\left( \frac{|-cz+a|}{-cz+a}\right) ^{2n+\alpha +1}F(\gamma ^{-1}(z))\right)
\\
& =\chi _{\Omega }(\gamma (z))\cdot F(z) \\
& =\chi _{\gamma ^{-1}(\Omega )}(z)F(z)=Q_{\gamma ^{-1}(\Omega )}F(z),
\end{align*}%
and $\Omega ^{0}\cap \gamma ^{-1}(\Omega )^{0}=\emptyset $. The second
property then immediately follows as
\begin{equation*}
\sum_{\gamma \in \Gamma }\tau _{n}^{\alpha }(\gamma )Q_{\Omega }\tau
_{n}^{\alpha }(\gamma ^{-1})=\sum_{\gamma \in \Gamma }Q_{\gamma ^{-1}(\Omega
)}=\sum_{\gamma \in \Gamma }\chi _{\gamma ^{-1}(\Omega )}=\chi _{\mathbb{C}%
^{+}}=I_{L^{2}(\mathbb{C}^{+},\mu )}.
\end{equation*}%
Repeating the arguments that we used to derive \eqref{eq:p-commutes} for $%
F=P_{\mathcal{W}_{\psi }}G$, $G\in L^2({\mathbb{C}}^+,\mu),$ then shows that 
$P_{\mathcal{W}_{\psi_n^\alpha }}\tau _{n}^{\alpha }(\gamma )=\tau
_{n}^{\alpha }(\gamma )P_{\mathcal{W}_{\psi_n^\alpha }}$. \hfill $\Box $%
\newline

\noindent Combining this with Proposition~\ref{dim}, we obtain%
\begin{equation*}
\text{dim}_{vN_{\omega }(\Gamma )}\mathcal{W}_{\psi _{n}^{\alpha }}=\text{dim%
}_{vN_{\omega }(\Gamma )}P_{\mathcal{W}_{\psi _{n}^{\alpha }}}L^2({\mathbb{C}%
}^+,\mu)=\text{trace}_{B(L^{2}({\mathbb{C}}^{+},\mu ))}(Q_{\Omega }P_{%
\mathcal{W}_{\psi _{n}^{\alpha }}}Q_{\Omega }),
\end{equation*}
which enables us to compute the von Neumann dimension of the wavelet spaces in terms of
the trace of the trace class integral operator $Q_{\Omega }P_{%
\mathcal{W}_{\psi _{n}^{\alpha }}}Q_{\Omega }$.

\begin{proposition}
\label{prop:dim} Let $\Gamma ,\Omega,Q_{\Omega}$ and $P_{\mathcal{W}%
_{\psi_n^\alpha }}$ as in Lemma~\ref{lem:proj-props}. Then 
\begin{equation*}
\emph{trace}_{B(L^2({\mathbb{C}}^+,\mu))}(Q_{\Omega}P_{\mathcal{W}%
_{\psi_n^\alpha }}Q_{\Omega})=|\Omega|\frac{\Vert \psi_n^\alpha \Vert
_{2}^{2}}{C_{\psi_n^\alpha }}.
\end{equation*}
\end{proposition}

\begin{proof}
First, observe that $\widetilde{P}:=Q_{\Omega}P_{\mathcal{W}_{\psi_n^\alpha
}}Q_{\Omega}$ is an integral operator with integral kernel 
\begin{equation*}
K_n^\alpha(z,w):=\chi _{\Omega}(z)k_{\psi_n^\alpha }(z,w)\chi_\Omega (w).
\end{equation*}%
If we show that $\widetilde{P}$ is trace class, then the result follows from 
\begin{equation*}
\text{trace}_{B(\mathcal{H})}(\widetilde{P})=\int_{\mathbb{C}%
^{+}}K_n^\alpha(z,z)d\mu (z)=\int_{\Omega}k_{\psi_n^\alpha }(z,z)d\mu
(z)=|\Omega|\frac{\Vert \psi_n^\alpha \Vert _{2}^{2}}{C_{\psi_n^\alpha }}.
\end{equation*}%
Note that the operator $Q_{\Omega}P_{\mathcal{W}_{\psi_n^\alpha }}$ is
Hilbert-Schmidt since 
\begin{align*}
\Vert Q_{\Omega}P_{\mathcal{W}_{\psi_n^\alpha }}\Vert _{{\mathcal{HS}}}^{2}&
=\int_{\mathbb{C}^{+}}\int_{\mathbb{C}^{+}}|\chi
_{\Omega}(z)k_{\psi_n^\alpha }(z,w)|^{2}d\mu (w)d\mu (z) \\
& =\int_{\Omega}\int_{\mathbb{C}^{+}}\frac{1}{C_{\psi_n^\alpha }}|\langle
\pi (z)\psi_n^\alpha ,\pi (w)\psi_n^\alpha \rangle |^{2}d\mu (w)d\mu (z) \\
& =\int_{\Omega}\Vert \pi (z)\psi_n^\alpha \Vert _{2}^{2}d\mu
(z)=|\Omega|\Vert \psi_n^\alpha \Vert _{2}^{2}<\infty .
\end{align*}%
It thus follows, that $\widetilde{P}$ is trace class as the product of two
Hilbert-Schmidt operators.
\end{proof}

\noindent With everything in place, the proofs of our main results take only
a few lines. 

\subsection{ Proof of Theorem~\ref{thm:1}}
If $\{\tau_n^\alpha (\gamma)F\}_{\gamma \in \Gamma }$ is a frame for $\mathcal{W}_{\psi_n^\alpha}$, then, as shown in \cite[Section 5]{KL}, this requires $%
\Omega $ to be compact. Moreover, as $\{\tau_n^\alpha (\gamma)F\}_{\gamma \in \Gamma }$ is in particular complete, we have that $vN_{\omega }(\Gamma )F$ is
dense in $\mathcal{W}_{\psi _{n}^{\alpha }}$. In other words, $F$ is a cyclic vector for $vN_{\omega
}(\Gamma )$. Theorem~\ref{thm-cyclic} thus shows that 
\begin{equation*}
\text{dim}_{vN_{\omega }(\Gamma )}\mathcal{W}_{\psi _{n}^{\alpha }}\leq 1.
\end{equation*}%
On the other hand, Proposition~\ref{prop:dim} yields 
\begin{equation*}
\text{dim}_{vN_{\omega }(\Gamma )}\mathcal{W}_{\psi _{n}^{\alpha }}=\frac{%
\Vert \psi _{n}^{\alpha }\Vert _{2}^{2}}{C_{\psi _{n}^{\alpha }}}|\Omega |%
\text{.}
\end{equation*}%
We conclude that%
\begin{equation*}
\left\vert \Omega \right\vert \leq \frac{C_{\psi _{n}^{\alpha }}}{\Vert \psi
_{n}^{\alpha }\Vert _{2}^{2}}.
\end{equation*}%
If $\{\tau
_{n}^{\alpha }(\gamma)F\}_{\gamma \in \Gamma }$ is a   Riesz sequence for $%
\mathcal{W}_{\psi_n^\alpha}$, then by \cite[Proposition~5.2]{RV} (the
statement assumes Kleppner's condition, which is implied by the ICC
condition), $F$ is separating for 
$vN_{\omega }(\Gamma )$. Theorem~\ref{thm-separating} now shows that 
\begin{equation*}
\text{dim}_{vN_{\omega }(\Gamma )}\mathcal{W}_{\psi _{n}^{\alpha }}\geq 1.
\end{equation*}%
We then conclude as before. \hfill$\Box$

\subsection{ Proof of Corollary~\ref{cor:other-family}}
By the orthogonality relation \eqref{eq:orth-rel} we have that $\{\varphi_\gamma\}_{\gamma\in\Gamma}$ is a frame (resp. Riesz sequence) for $\H^2(\C^+)$ if and only if $\{W_{\psi_n^\alpha}\varphi_\gamma\}_{\gamma\in\Gamma}$ is a frame (resp. Riesz sequence) for $\mathcal{W}_{\psi_n^\alpha}$. Using the definition of $\varphi_\lambda$ in \eqref{eq:def-varphi} to calculate $W_{\psi_n^\alpha}\varphi_\gamma$ we find
\begin{align*}
W_{\psi_n^\alpha}\varphi_\gamma(z)&=\int_{\C^+} \Phi(w) \left(\frac{cw+d}{|cw+d|}\right)^{2n+\alpha+1} \langle \pi(\gamma(w))\psi_n^\alpha,\pi(z)\psi_n^\alpha\rangle d\mu(w)
\\
&=\left(\frac{|-cz+a|}{-cz+a}\right)^{2n+\alpha+1}\int_{\C^+} \Phi(w)  \langle \pi(w)\psi_n^\alpha,\pi(\gamma^{-1}(z))\psi_n^\alpha\rangle d\mu(w)
\\
&=\left(\frac{|-cz+a|}{-cz+a}\right)^{2n+\alpha+1}P_{\mathcal{W}_{\psi_n^\alpha}}\Phi(\gamma^{-1}(z))
\\
&=\tau_n^\alpha(\gamma^{-1})P_{\mathcal{W}_{\psi_n^\alpha}}\Phi(z),
\end{align*}
by \eqref{eq:identity-k-psi}. Hence, $P_{\mathcal{W}_{\psi_n^\alpha}}\Phi$ is a cyclic (resp. separating) vector for $\mathcal{W}_{\psi_n^\alpha}$ and the result follows from Theorem~\ref{thm:1}. \hfill$\Box$

\subsection{ Proof of Corollary~\ref{cor:wavelet}} 
If $\{\pi (\gamma (z))\psi
_{n}^{\alpha }\}_{\gamma \in \Gamma }$ is a wavelet frame (resp. Riesz sequence) for $\mathcal{H}%
^{2}(\mathbb{C}^{+})$, then $\left\{ \overline{k_{\psi _{n}^{\alpha
}}(\gamma (z),\cdot )}\right\} _{\gamma \in \Gamma }$ generates a frame (resp. Riesz sequence) for $%
\mathcal{W}_{\psi _{n}^{\alpha }}$.   Moreover, since $\Gamma $ is a subgroup of $PSL(2,%
\mathbb{R})$, it follows from (\ref{eq:identity-k-psi}) that 
\begin{equation*}
\left\{ \tau _{n}^{\alpha }(\gamma )\overline{k_{\psi _{n}^{\alpha
}}(z,\cdot )}\right\} _{\gamma \in \Gamma }=\left\{ \left( \frac{cz+d}{|cz+d|%
}\right) ^{2n+\alpha +1}\overline{k_{\psi _{n}^{\alpha }}(\gamma (z),\cdot )}%
\right\} _{\gamma \in \Gamma }.
\end{equation*}%
As the factor in front of $\overline{k_{\psi _{n}^{\alpha }}(\gamma (z),w )}$ is unimodular, and independent of $w$, it follows that 
$\overline{
k_{\psi _{n}^{\alpha }}(z,\cdot )}$ is a cyclic (resp. separating) vector for $vN_{\omega
}(\Gamma )$ and the first and third statements follow from Theorem~\ref{thm:1}.
For the second statement, assume that $\left\vert \Omega \right\vert >{%
C_{\psi _{n}^{\alpha }}}/{\Vert \psi _{n}^{\alpha }\Vert _{2}^{2}}$, then
there exist no cyclic vector in $\mathcal{W}_{\psi _{n}^{\alpha }}$. In
particular, there exist $F\in W_{\psi _{n}^{\alpha }}$ that is orthogonal to 
$vN_{\omega }(\Gamma )\overline{k_{\psi _{n}^{\alpha }}(z,\cdot )}$, i.e.
for every $\gamma \in \Gamma $ 
\begin{align*}
0&= \langle F,\tau _{n}^{\alpha }(\gamma )\overline{k_{\psi _{n}^{\alpha
}}(z,\cdot )}\rangle   = \langle \tau _{n}^{\alpha }(\gamma ^{-1})F,\overline{k_{\psi _{n}^{\alpha
}}(z,\cdot )}\rangle= \tau_n^\alpha(\gamma^{-1})F(z),
\end{align*}%
which in turn implies $F(\gamma(z))=0$ for every $\gamma\in\Gamma$.
\hfill $\Box $

\subsection{Proof of Theorem~\ref{thm:super-wavelet}}
For the choice $\alpha=2(B-n)-1$ we have that 
the exponent of the phase factor in \eqref{eq:tau-n} is given by $2B$ (i.e., it is independent of $n=0,1,...,\lfloor B-\frac{1}{2}\rfloor$). By Lemma~\ref{lem:invariance}, each space $A_n^B(\C^+,\mu)$ is invariant under the projective representation $\sigma_{2B}$. Consequently,  the orthogonal sum   $\boldsymbol{A}_N^B(\C^+,\mu)=\bigoplus_{n=0}^{N-1} A_n^B(\C^+,\mu)$   is also invariant under $\sigma_{2B}$. One can then proceed in a similar fashion as in the proofs of Theorem~\ref{thm:1}. In particular, the reproducing kernel of $\boldsymbol{A}_N^B(\C^+,\mu)$ is given by $\boldsymbol{k}_N^B=\sum_{k=0}^{N-1} k_{\psi_n^{2(B-n)-1}}$ with diagonal
\begin{align*}
\boldsymbol{k}_{N}^B(z,z)&=\sum_{n=0}^{N-1}  {k}_{ {\psi}^{2(B-n)-1}_n}(z,z)  =\sum_{n=0}^{N-1} \frac{2(B-n)-1}{2} =\frac{N(2B-N)}{2},\quad  z\in\C^+.
\end{align*}
It remains to note that (as in Lemma~\ref{lem:proj-props}) the projection $P_{\boldsymbol{A}_N^B(\C^+,\mu)}$ commutes with $\sigma_{2B}(\gamma)$, for every $\gamma\in\Gamma$, which implies (using the same steps as in the proof of Proposition~\ref{prop:dim}) that
$$
\text{dim}_{vN_{\omega }(\Gamma )}\boldsymbol{A}_N^B(\C^+,\mu) =|\Omega|\boldsymbol{k}^B_N(i,i)=|\Omega|\frac{N(2B-N)}{2}.
$$
\pbox

\end{document}